\documentclass[12pt,a4paper]{amsart}
\usepackage{amssymb}
\usepackage{amsfonts}
\usepackage{amsmath}
\usepackage[mathscr]{eucal}

\usepackage{color}

\newtheorem{thm}{Theorem}

\newtheorem{lem}[thm]{Lemma}
\newtheorem{cor}[thm]{Corollary}

\numberwithin{equation}{section}

\def\N{{\Bbb N}}
\def\Z{{\Bbb Z}}

\def\C{{\Bbb C}}

\def\SL{{\operatorname{SL}}}

\def\PSL{{\operatorname {PSL}}}

\def\tr{{\operatorname{tr}}}

\def\leq{\leqslant}
\def\geq{\geqslant}

\def\le{\leq}
\def\ge{\geq}

\def\d {{{d}}}

\def\1{{\bold 1}}

\renewcommand{\a}{\alpha}
\renewcommand{\b}{\beta}

\newcommand{\e}{\epsilon}
\newcommand{\f}{\varphi}

\renewcommand{\l}{\lambda}

\newcommand{\s}{\sigma}

\newcommand{\go}{{\mathfrak{o}}}

\newcommand{\Ocal}{{\mathcal O}}

\newcommand{\CC}{\mathbb{C}}

\newcommand{\NN}{\mathbb{N}}

\newcommand{\QQ}{\mathbb{Q}}

\newcommand{\ZZ}{\mathbb{Z}}

\title{Optimal estimates for an average of Hurwitz class numbers}

\author{Shingo Sugiyama}
\author{Masao Tsuzuki}

\begin{document}

	\begin{abstract}
		In this paper, we give
		an optimal estimate of an average of Hurwitz class numbers.
		As an application, we give an equidistribution result of the family $\{\frac{t}{2q^{\nu/2}} \ | \ \nu \in \NN, t \in \ZZ, |t|<2q^{\nu/2}\}$
		with $q$ prime, weighted by Hurwitz class numbers.
		This equidistribution produces
		many asymptotic relations among Hurwitz class numbers.
		Our proof relies on
		the resolvent trace formula of Hecke operators
		on elliptic cusp forms of weight $k\ge 2$.
	\end{abstract}
	\setcounter{tocdepth}{1}

\maketitle
%%%%%%%%%%%%%%%%%%%%%%%%%%%%%%%%%%%%%%%%%	
\section{Introduction} \label{Sec1}
%%%%%%%%%%%%%%%%%%%%%%%%%%%%%%%%%%%%%%%%%

For a positive integer $D>0$ with $D\equiv 0,3 \,(\text{mod } 4)$,
let us define $H(D)$ as 
$$H(D)=\sum_{Q}\frac{1}{\#\Gamma_Q},$$
where $Q$ runs over $\PSL_2(\ZZ)$-equivalence classes of positive 
definite integral binary quadratic forms with discriminant $-D$,
and $\Gamma_Q$ denotes the stabilizer of $Q$ in $\PSL_2(\ZZ)$.
The positive rational number $H(D)$ is commonly called the $D$-th {\it Hurwitz class number} or {\it Kronecker-Hurwitz class number}.
There are some relations among $\{H(D)\}_D$: 
\begin{align}\label{rel of H prime}
\sum_{\substack{t\in \ZZ \\ t^2<4q}}H(4q-t^2)=2q
\end{align}
for a prime number $q$,
or more generally,
Hurwitz's formula
\begin{align}
\sum_{\substack{t\in \Z \\ t^2<4m}} H(4m-t^2)=2\sigma(m)-\sum_{0<d|m}\min(d,m/d) \label{HWformula}
\end{align}
for $m\in \NN$, where $\s(m)$ is the divisor function (\cite{Kronecker}, \cite{Gierster}, \cite{Hurwitz}).
In \cite{Zagier0} and \cite{HirzeburchZagier}, it was shown that the generating function of $\{H(D)\}_D$ is a mock modular form of weight $3/2$ on $\Gamma_0(4)$, and this discovery inspired
many mathematicians to produce new relations of the Hurwitz class numbers
via modular forms.
Relations \eqref{rel of H prime} with condition
``$t\equiv c (\text{mod } a)$'' were given in \cite{BCFJS}
for $a=2, 3, 4$, and conjectured for $a=5, 7$.
Later, Bringmann and Kane \cite{BringmannKane} proved the conjecture for $a=5, 7$
by using mixed mock modular forms.
In recent years, many other relations were revealed by Mertens \cite{Mertens}, \cite{Mertens2}. 

For $m\in \N$, let us consider a discrete measure 
$$\mu_{m} = \sum_{\substack{t\in\ZZ \\ t^2<4m}}H(4m-t^2)\delta_{\frac{t}{2\sqrt{m}}}$$
on the interval $[-1,1]$, where $\delta_a$ is the Dirac measure supported at $a$. Observing that the total mass $\langle\mu_{m},1\rangle$ is exactly given by the formula \eqref{HWformula}, one may naturally raise the following two questions concerning the measure $\mu_{m}$: 
\begin{enumerate}
\item Can one give an exact formula of the $n$-th moment  $\langle \mu_m, x^n \rangle$ for any $n \in \ZZ_{\ge 0}$ ? 
\item What can one say about the limiting behavior of $\mu_m$ as $m\rightarrow \infty$ ? 
\end{enumerate}
In a sence, the first question (1) has been completely answered by the Eichler-Selberg trace formula of Hecke operators on the space of cusp forms on $\SL_2(\Z)$ due to the fact that every monomial $x^{n}$ is a linear combination of Chebyshev polynomials of the 2nd kind. Indeed, \eqref{HWformula} is identical to the trace formula of the $m$-th Hecke operator $T(m)$ on the space of weight $2$ modular forms $S_2(\SL_2(\Z))=\{0\}$. For a prime number $q$, the Eichler-Selberg trace formulas of $T(q^{\nu})$ $(\nu=0,1,\dots)$ on $S_k(\SL_2(\Z))$ are collectively written by using a generating series: 
\begin{thm}\label{SL2RTF} Let $k\geq 4$ be an even integer and $S_k(\SL_2(\Z))$ the space of cusp forms of weight $k$ on $\SL_2(\Z)$. Let $\{f_i\}$ be the normalized Hecke eigen basis of $S_k(\SL_2(\Z))$ and $f_i(\tau)=\sum_{n=1}^{\infty}a_i(n)e^{2\pi i n\tau}$ its Fourier expansion at $i\infty$. Let $q$ be a prime number. Then we have the following identity in the formal power series ring of indeterminate $X$: \begin{align}
\sum_{i} \frac{1}{q^{(1-k)/2}a_i(q)-(X+X^{-1})}
&=-\frac{k-1}{12}\frac{X}{1-q^{-1}X^2}
-\frac{1}{2} \frac{X}{1-X^2}
+\frac{X}{(1-X^2)(1-q^{\frac{1-k}{2}}X)} 
 \label{RTF1}
\\
& \quad +\frac{X}{2}\sum_{\nu=0}^{\infty}\langle \mu_{q^{\nu}},U_{k-2} \rangle \,(q^{-1/2}X)^{\nu},
 \notag
\end{align}
where $U_{l}(x)$ is the Chebyshev polynomial of the 2nd kind of degree $l$.
\end{thm}
The formula \eqref{RTF1} is a simplest case of the resolvent trace formula \eqref{RTF} of Hecke operators explained in \S~\ref{Rtf}.
In this paper, we prove the following theorem, which answers the question (2) at least when $m$ is a power of a fixed prime number $q$.
\begin{thm}\label{equidist}
	Let $q$ be a prime number. For any $\a,\b \in [-1,1]$ with $\a<\b$, we have
	\begin{align*}
	\lim_{\nu\rightarrow \infty}\frac{{\mu}_{q^\nu}([\a, \b])}{\mu_{q^\nu}([-1,1])} =\lim_{\nu\rightarrow \infty}
	\frac{1-q^{-1}}{2q^\nu}\sum_{\substack {t\in \Z \\\a\leq \frac{t}{2q^{\nu/2}}\leq \b}} H(4q^{\nu}-t^2) = \frac{2}{\pi}\int_{\a}^{\b} \sqrt{1-x^2}\,\d x.
	\end{align*}
\end{thm}
Note that this theorem exhibits a weighted equidistribution of points $\{\frac{t}{2q^{\nu/2}}\}_{,\nu\in\Z_{\ge 0},\ |t|<2q^{\nu/2}}$ in $[-1,1]$ with weight factor $\{H(4q^\nu-t^2)\}_{\nu, t}$. The proof of this theorem relies on a good estimation of the $n$-th moments $\langle \mu_{q^\nu},x^{n}\rangle$ as $\nu\rightarrow \infty$ for each $n$. For $n\in \ZZ_{\ge 0}$ and $\a>0$, let us consider the following statement asserting a bound of the $n$-th moments of $\mu_{q^{\nu}}$ by $(q^\nu)^{\a+\e}$: 
$$({\rm E}_{n,\a}) : \,(\forall \e>0)\,(\exists C>0)\,(\forall \nu\in \N)\,
\left|\sum_{\substack{t \in \ZZ\\ t^2<4q^{\nu}}} H(4q^\nu-t^2)\left(\frac{t}{2 \sqrt{q^{\nu}}}\right)^n\right| <C q^{\nu(\a+\e)}.
$$
The trivial bound $({\rm E}_{n,1})$ is easily obtained by the class number formula and the upper bound $|L(1,\chi)|\ll \log |D|$ of the Dirichlet $L$-functions
associated with non-principal Dirichlet characters $\chi$ modulo $D$ (cf.\ \eqref{trivial bound}). We have the following optimal improvement of the trivial bound.  
\begin{thm}\label{best possible}
Let $q$ be a prime number. The statement $({\rm E}_{n,1/2})$ holds for all $n\in \N$. If $0<\a<1/2$, there exists $n\in \N$ such that $({\rm E}_{n,\a})$ does not hold. 
\end{thm}
The proof of Theorem~\ref{best possible} in turn relies on the resolvent trace formula \eqref{RTF1}.

%In this paper,
%we give an optimal estimate of an average of Hurwitz class numbers
%weighted by monomials as follows.
%(The meaning of ``optimal'' is explained later.)
%\begin{thm}\label{best possible}
%Let $q$ be a prime number. Then,
%$(E_{q,1/2})$ holds true and $(E_{\a})$ is false for any $\a<<1/2$.
%
%Let $\nu$ be a non-negative integer. For any $m\in \NN$, we have the bound
%$$\left|\sum_{t \in \ZZ, \ t^2<4q^\nu}H(4q^{\nu}-t^2)\left(\frac{t}{2q^{\nu/2}}\right)^m\right| \ll_{q,\e, m} q^{(1/2+\e)\nu}$$
%with the implied constant independent of $\nu$.
%Here $\ll$ is the Vinogradov notation.
%\end{thm}

%%%%%%%%%%%%%%%%%%%%%%%%%%%%%%%%%%%%%%%%%
\section{Resolvent trace formulas} \label{Rtf}
%%%%%%%%%%%%%%%%%%%%%%%%%%%%%%%%%%%%%%%%%

%Trace formulas have been developped as powerful tools in both representation theory and number theory.
%Among those, the Eichler-Selberg trace formula of a Hecke operator $T(m)$ on the space $S_k(N, \chi)=S_k(\Gamma_0(N),\chi)$ of elliptic cusp forms of weight $k$, level $N$ and nebentypus $\chi$ 
%is the simplest and fundamental example, which provides us
%much information: the dimension formula of $\dim S_k(N, \chi)$,
%the vertical Sato-Tate law \cite{Serre} on Fourier coefficients of elliptic cusp forms, and so on.
%In this article, we give the resolvent trace formula of the $q$-th Hecke operator $T(q)$
%with $q$ being a prime number (see Theorem \ref{resolventTF}).
%As an application, we give a limit formula on the sum of the Hurwitz class numbers $H(D)$ (see Corollary \ref{equidist}).

 We shall state the {\rm resolvent trace formula} in a greater generality for Hecke operators acting on the space $S_k(N, \chi)=S_k(\Gamma_0(N),\chi)$ of elliptic cusp forms of weight $k$, level $N$ and of nebentypus $\chi$, where $k$ and $N$ are positive integers and $\chi$ is a Dirichlet character modulo $N$. The conductor of $\chi$ is denoted by $f_\chi$. For $m\in \NN$ relatively prime to $N$, the $m$-th Hecke operator $T(m)$ on $S_k(N, \chi)$
is
defined as
$$
 T(m)f(z) = m^{k-1}\sum_{\substack{a,d \in \NN \\ ad=m}}\sum_{b=0}^{d-1}\frac{\chi(a)}{d^{k}}f\left(\frac{az+b}{d}\right).
$$
The trace formula for $T(m)$ is well-known and can be stated in a couple of different forms. Here we quote the formula from \cite[Theorem 2.2]{Schoof}, which is convenient for our purpose. 

\begin{thm}[Eichler-Selberg trace formula]
	\label{ES trace}
Assume $k\ge 2$ and $\chi(-1)=(-1)^k$.
	Let $m$ be a positive integer relatively prime to $N$. We have 
	$${\tr}(T(m)|S_k(N,\chi))=A_1(m)+A_2(m)+A_3(m)+A_4(m),$$
	where $A_1$ is defined by
	$$A_1(m)=m^{k/2-1}\tilde{\chi}(\sqrt{m})\frac{k-1}{12}\psi(N)$$
	with $\psi(N)=N\prod_{p|N}(1+p^{-1})$, and $\tilde{\chi}(\sqrt{m})=\chi(\sqrt{m})$ if $m$ is a square, and $0$ otherwise.
	The term $A_2(m)$ is defined by
	$$A_2(m)=-\frac{1}{2}\sum_{\substack{t \in \ZZ \\ t^2<4m}}m^{\frac{k-2}{2}}U_{k-2}\left(\frac{t}{2\sqrt{m}}\right)
	H_{N,\chi}(4m-t^2),$$
	where
	%$$P_{k,1}(t,m)=\frac{\eta^{k-1}-\overline{\eta}^{k-1}}{\eta-\overline{\eta}}
	%=m^{\frac{k-2}{2}}U_{k-2}\left(\frac{t}{2\sqrt{m}}\right)$$
	%$$P_{k,1}(t,m)
	%=m^{\frac{k-2}{2}}U_{k-2}\left(\frac{t}{2\sqrt{m}}\right)$$ and
	$U_n(x)$ is the $n$-th Chebyshev polynomial of 2nd kind. The value $H_{N,\chi}(4m-t^2)$ is defined as
	$$H_{N, \chi}(4m-t^2)=\sum_{\substack{f \in \NN \\ f^2 | t^2-4m \\ \frac{t^2-4m}{f^2}\equiv 0,1 \pmod 4 }} h_w\left(\frac{t^2-4m}{f^2}\right)\mu(t,f,m),$$
	where $h_w(D)=\frac{2h_D}{\#\go_D^\times}$ with $\go_D$ and $h_D$
	being the order in $\QQ(\sqrt{D})$ of conductor $D<0$ and
	its class number,
	$$\mu(t,f,m) =\frac{\psi(N)}{\psi(N/N_f)}\sum_{\substack{x \in (\ZZ/N \ZZ)^\times
			\\ x^2-tx+m \equiv 0 \pmod  {N N_f}}}\chi(x)$$
	with $N_f=\gcd(N,f)>0$.
	The term $A_3$ is defined by
	$$A_3(m)=-\sideset{}{'}\sum_{\substack{d|m \\ 0<d\le \sqrt{m}}}d^{k-1}\sum_{\substack{0<c|N
			\\ \gcd(c,N/c)|\gcd(N/f_\chi, m/d-d)}}\varphi(\gcd(c, N/c))\chi(y).$$
	Here we set $\sideset{}{'}\sum_{d|m, 0<d\le \sqrt{m}}F(d)=\sum_{d|m, 0<d< \sqrt{m}}F(d) + \frac{1}{2}F(\sqrt{m})$
	for a function $F : \NN \rightarrow \CC$ with $F(\sqrt{m})=0$ unless $\sqrt{m}\in\ZZ$,
	an element $y=y_{m,N, c,d} \in \ZZ/(N/\gcd(c,N/c))\ZZ$ is taken so that $y\equiv d \ ({\rm mod}\ c)$ and $y\equiv m/d \ ({\rm mod}\ N/c)$, and
	the symbol $\varphi$ means the Euler totient function.
	The term $A_4$ is defined as
	$$A_4(m) = \delta_{k,2}\delta_{\chi,\bf1}\sum_{0<t|m}t.$$
\end{thm}
Let us explain the resolvent trace formula of Hecke operators. For simplicity, we restrict ourselves to the simplest case involving only one prime $q$ relatively prime to $N$. We fix a square root $\chi(q)^{1/2}\in \CC^\times$ and set $\chi(q)^{-1/2}=(\chi(q)^{1/2})^{-1}$ and $\chi(q^n)^{1/2}=\chi(q)^{n/2}$; although there are two choices of square roots, it does not influence the argument throughout this article. Set $T'(q)=\chi(q)^{-1/2}q^{(1-k)/2}T(q)$ and consider
$$\tr((T'(q)-\lambda)^{-1}|S_k(N,\chi)),
$$ where $\l \in \CC$ is not equal to any eigenvalues of $T'(q)$ on $S_k(N,\chi)$. Let $\{ f_i\}$ be the orthogonal basis consisting of normalized Hecke eigenforms of $S_{k}(N,\chi)$ with respect to the Petersson inner product.
Let $f_i(\tau)=\sum_{n=1}^\infty a_i(n)e^{2\pi i n\tau}$ be the Fourier expansion of $f_i$ at $i \infty$.
We remark that $a_{i}(q)=\chi(q)\overline{a_{i}(q)}$. The parameter $\{\a_i(q), \a_i(q)^{-1}\}$ is defined by $\chi(q)^{-1/2}q^{(1-k)/2}a_{i}(q) = \a_i(q)+\a_i(q)^{-1}$. Then, the Satake parameter of $f_i$ at $q$ is written as $\{\chi(q)^{1/2}\a_i(q), \chi(q)^{1/2}\a_i(q)^{-1}\}$. By the Ramanujan-Petersson conjecture (Deligne's theorem \cite[${\rm N}^\circ$ 5]{Deligne1}, \cite{Deligne2}), we have $|\a_i(q)|=1$, or equivalently, $\chi(q)^{-1/2}q^{(1-k)/2}a_{i}(q) \in [-2,2]$.
Furthermore, we have $|a_{i}(n)| \le d(n)n^{(k-1)/2}$ for any $n \in \NN$ relatively prime to $N$, where $d(n)$ is the number of the positive divisors of $n$. Since $S_{k}(N,\chi)$ is a finite dimensional $\C$-vector space, the operator $T'(q)-\lambda$ for $\lambda\in \C$ with large enough absolute value is invertible.
%Then the trace $\tr((T'(q)-\lambda)^{-1}| S_k(N,\chi))$ is equal to
%\begin{align*}
%\sum_{i} \frac{1}{\a_i(q)+\a_i(q)^{-1}-\lambda}= \sum_{i} %\frac{-X}{(X-\a_i(q))(X-\a_i(q)^{-1})},
%\end{align*}
%where $X \in \CC$ is given by $\l=X+X^{-1}$.

For $c\in \NN$ such that $c|N$,
set $l=\gcd(c,N/c)$.
Then, $c_1:=\gcd(c,f_\chi)$ and $c_2:=c/c_1$
satisfy
$c_1|f_\chi$, $c_1|c$, $\gcd(c,f_\chi/c_1)=1$ and
$\gcd(c_2, f_\chi)=1$.
By $\gcd(c_1, f_\chi/c_1)=1$,
we have the decomposition $\chi=\chi_c\times \chi_c'$
of Dirichlet characters so that
$f_{\chi_c}=c_1$ and $f_{\chi_c'}=f_\chi/c_1$.

For $q,l \in \NN$, let $m_{q,l}$ denote the order of $q$ in $(\ZZ/l\ZZ)^\times$.

The resolvent trace formula of $T(q)$ is described by $\chi_c'$ and $m_{q,l}$ as follows.
\begin{thm}[The resolvent trace formula of $T(q)$] \label{resolventTF}
Let $q$ be a prime number not dividing $N$.
If $|X|$ is sufficiently small, then $\tr (\{T'(q) - (X+X^{-1})\}^{-1}|S_k(N,\chi))$
is equal to
	\begin{align}\label{RTF}
	&  \sum_{i}\frac{1}{\a_i(q)+\a_i(q)^{-1}-(X+X^{-1})}\\
	= \,& -\frac{k-1}{12}\psi(N)\frac{X}{1-q^{-1}X^2}
	-\frac{1}{2}\frac{X}{1-X^2}\sum_{\substack{0<c|N \\ \gcd(c,N/c)|(N/f_\chi)}}
	\varphi(\gcd(c,N/c)) \notag\\
	&+\sum_{\substack{0<l|(N/f_\chi) }}\varphi(l)\sum_{\substack{0<c | N \\ l=\gcd(c,N/c)}}
	\frac{X}{(1-X^2)(1-\{\chi_c'(q)\chi(q)^{-1/2}q^{(1-k)/2}X\}^{m_{q,l}})} \notag \\
	%\sum_{\nu=0}^{\infty}A_3(q^\nu)(\chi(q)^{-1/2}q^{(1-k)/2}X)^{\nu}
	\notag \\
	& -\delta_{k,2}\delta_{\chi,{\bf1}}\frac{X}{(1-q^{1/2}X)(1-q^{-1/2}X)}
	+\frac{X}{2}\sum_{\nu=0}^{\infty}A_{k,q}(\nu)(\chi(q)^{-1/2}q^{-1/2}X)^{\nu},\notag
	\end{align}
		where we set
	$$A_{k,q}(\nu)=\sum_{\substack{t \in \ZZ \\ t^2<4q^\nu}}H_{N,\chi}(4q^\nu-t^2)U_{k-2}\left(\frac{t}{2q^{\nu/2}}\right).$$
\end{thm}
%Theorem \ref{resolventTF} gives explicit formulas of
%$n$-th moments $\langle\mu_{q^\nu}, x^n\rangle$
%by means of Hecke eigenvalues of $T(q^r)$ $(1\le r\le \nu)$
%with the aid of the comparison of the coefficient of $X^\nu$.

%%%%%%%%%%%%%%%%%%%%%%%%%%%%%%%%%%%%%%%%%%%%%%%%%%%%%
\section{Proof of Theorem \ref{resolventTF}}
%%%%%%%%%%%%%%%%%%%%%%%%%%%%%%%%%%%%%%%%%%%%%%%%%%%%%
Theorem~\ref{resolventTF} is deduced from the trace formulas of $T(q^{\nu})$ with $\nu\in \N$ recalled in Theorem~\ref{ES trace} by considering the generating series
\begin{align}
\sum_{\nu=0}^{\infty}\tr(T(q^\nu)|S_k(N,\chi))\zeta^{\nu+1}
 \label{GENseries}
\end{align}
with $X=\chi(q)^{1/2}q^{(k-1)/2}\zeta$. Although the proof is completely elementary, we include it for convenience of the readers. 

Invoking $U_n(\cos \theta)=\sin((n+1)\theta)/\sin\theta$,
a direct computation gives us the following.
\begin{lem} \label{L1}
	Let $\a\in \C$. Then, for any $X \in \CC$ such that $|X|<\min(|\a|,|\a|^{-1})$,
	\begin{align*}
	\frac{1}{\a+\a^{-1}-(X+X^{-1})}=-\sum_{n=0}^{\infty} U_{n}\left(\frac{\a+\a^{-1}}{2}\right)\,X^{n+1}.
	\end{align*}
\end{lem}
To prove Theorem \ref{resolventTF},
	it is sufficient to the case $X=\chi(q)^{1/2}q^{(k-1)/2}\zeta$ with $|\zeta|$ sufficiently small.
	By noting $U_\nu(2^{-1}(\a_i(q)+\a_i(q)^{-1})) = \chi(q^{\nu})^{-1/2}q^{(1-k)\nu/2}a_{f_i}(q^\nu)$ ($\nu \in \ZZ_{\ge0}$) deduced from the recurrence equation of Hecke operators, we have that \eqref{GENseries} equals
	$$-\chi(q)^{-1/2}q^{\frac{1-k}{2}}\sum_i\frac{1}{\a_i(q)+\a_i(q)^{-1}-(X+X^{-1})}.$$
	A direct computation also reveals $$\sum_{\nu=0}^{\infty}A_1(q^\nu)\zeta^{\nu+1}=\frac{k-1}{12}\psi(N)\frac{\chi(q)^{-1/2}q^{(1-k)/2}X}{1-q^{-1}X^2}.$$
	and
	$$\sum_{\nu=0}^{\infty}A_2(q^\nu)\zeta^{\nu+1}=-\frac{1}{2}\sum_{\nu=0}^{\infty}A_{k,q}(\nu) (\chi(q)^{-1/2}q^{-1/2}X)^{\nu}\times\chi(q)^{-1/2}q^{(1-k)/2}X.$$
	The series as above is absolutely and locally uniformly convergent on $|X|\ll q^{-1-\e}$ for a small $\e>0$.
	Indeed, by invoking the relation
	$$h_w(D)=\frac{h_{D_0}}{\#\go_{D_0}^\times}f\prod_{p|f}(1-p^{-1}\left(\tfrac{D_0}{p}\right)), \qquad (D=D_0 f^2),$$
where $D_0<0$ is the fundamental discriminant, the class number formula
and the estimate $L(1, (\frac{D_0}{p}))\ll \log |D_0|$,
we obtain
	the inequality
	\begin{align}
	\label{trivial bound}
|A_{k,q}(\nu)|\ll_{k,N} \sum_{\substack{t\in\ZZ \\ t^2<4q^{\nu}}}H_{N,\bf1}(4q^\nu-t^2)
\ll_\e \sum_{\substack{t \in \ZZ \\
	t^2<4q^\nu}}\sum_{\substack{f \in \NN \\f^2 | (t^2-4q^\nu)}}|\tfrac{t^2-4q^\nu}{f^2}|^{1/2+\e}\ll_{\e}
(4q^\nu)^{1+2\e}.
\end{align}

	The series involving $A_4$ is computed as
	\begin{align*}
	\sum_{\nu=0}^{\infty} A_4(q^{\nu})\zeta^{\nu+1} = &\delta_{k,2} \delta_{\chi,\bf1}\sum_{\nu=0}^{\infty}
	\sum_{\substack{0<t|q^{\nu}}}t\zeta^{\nu+1}
	= \delta_{k,2}\delta_{\chi,\bf1}\sum_{\nu=0}^{\infty}
	\sum_{j=0}^{\nu}q^j\zeta^{\nu+1}\\
	= & \delta_{k,2}\delta_{\chi,\bf1}\frac{1}{q-1}\left(\frac{q\zeta}{1-q\zeta}-\frac{\zeta}{1-\zeta}\right)
	= \delta_{k,2}\delta_{\chi,\bf1}\frac{\zeta}{(1-q\zeta)(1-\zeta)} \\
	= & \delta_{k,2}\delta_{\chi,\bf1}
	\frac{q^{-1/2}X}{(1-q^{1/2}X)(1-q^{-1/2}X)}.
	\end{align*}
To complete the proof of Theorem \ref{resolventTF},
we give an explicit formula of
$$F(\zeta)=\sum_{\nu=0}^{\infty}A_3(q^{\nu})\zeta^\nu.$$

\begin{lem} \label{explicit of generating A3} Set $X=\chi(q)^{1/2}q^{(k-1)/2}\zeta$. If $|X|<1$,
	then $F(\zeta)$ is absolutely and locally uniformly convergent, and we have
\begin{align*}
F(\zeta)
	=& -\sum_{\substack{0<l|(N/f_\chi) }}\varphi(l)\sum_{\substack{0<c | N \\ l=\gcd(c,N/c)}}
\frac{1}{(1-X^2)(1-\{\chi_c'(q)\chi(q)^{-1/2}q^{(1-k)/2}X\}^{m_{q,l}})} \\
&+\frac{1}{2}\frac{1}{1-X^2}\sum_{\substack{0<c|N \\ \gcd(c,N/c)|(N/f_\chi)}}
\varphi(\gcd(c,N/c)).
\end{align*}
\end{lem}
\begin{proof}
We start the proof of the following expression of $F(\zeta)$:
{\allowdisplaybreaks\begin{align}
	& -\sum_{\nu=0}^{\infty}\sum_{j=0}^{\nu}q^{(k-1)j}
	\sum_{\substack{c|N \\
			\gcd(c,N/c)|\gcd(N/f_\chi, (q^{2\nu+1-2j}-1)q^j)}}
	\varphi(\gcd(c,N/c))\chi(y)\zeta^{2\nu+1} \label{first} \\ 
	& 
	-\sum_{\nu=0}^{\infty}\sum_{j=0}^{\nu}
	q^{(k-1)j}
	\sum_{\substack{c|N \\
			\gcd(c,N/c)|\gcd(N/f_\chi, (q^{2\nu-2j}-1)q^j)}}
	\varphi(\gcd(c,N/c))\chi(y)\zeta^{2\nu} \label{second} \\
	& +\sum_{\nu=0}^{\infty}\frac{1}{2}q^{(k-1)\nu}
	\sum_{\substack{c|N \\
			\gcd(c,N/c)|\gcd(N/f_\chi, 0)}}
	\varphi(\gcd(c,N/c))\chi(y) \zeta^{2\nu}. \label{third}
\end{align}
}The third term \eqref{third} equals
\begin{align}\frac{1}{2}\frac{1}{1-\chi(q)q^{k-1}\zeta^2}\sum_{\substack{c|N \\ \gcd(c,N/c)|(N/f_\chi)}}
\varphi(\gcd(c,N/c))
\label{third-formula}
\end{align}
for $|X|=|q^{(k-1)/2}\zeta|<1$ by $\chi(y)=\chi(q)^\nu$ and an easy calculation.

We consider the first term \eqref{first}.
%%Fix $c\in \NN$ such that
%%$$c|N, \qquad \gcd(c,N/c)|\gcd(N/f_\chi, q^{2\nu+1-2j}-1),$$
%%and set $l=\gcd(c,N/c)$.
%%Then, $c_1:=\gcd(c,f_\chi)$ and $c_2:=c/c_1$
%%satisfy
%%$c_1|f_\chi$, $c_1|c$, $\gcd(c,f_\chi/c_1)=1$ and
%%$\gcd(c_2, f_\chi)=1$.
%%By $\gcd(c_1, f_\chi/c_1)=1$,
%%we have the decomposition $\chi=\chi_c\times \chi_c'$
%%of Dirichlet characters so that
%%$f_{\chi_c}=c_1$ and $f_{\chi_c'}=f_\chi/c_1$.
%Take $u, v \in \ZZ$ such that $uc/l + vN/(cl)=1$.
%We may assume $y=q^j vN/(cl) +q^{2\nu+1-j}uc/l$ because of $q^j\equiv q^{2\nu+1-j} (\text{mod }l)$.
Noting $y\equiv q^j (\text{mod } c_1)$ and
$y\equiv q^{2\nu+1-j} (\text{mod } f_\chi/c_1)$, we have
\begin{align*}
	\chi(y)=\chi_c(q^j)\chi_c'(q^{2\nu+1-j})=\chi_c(q)^j\chi_c'(q)^{2\nu+1-j}.
\end{align*}
From this, \eqref{first} is rewritten as
\begin{align*}
	-\sum_{\nu=0}^{\infty}
	\sum_{j=0}^{\nu}q^{(k-1)j}
	\sum_{\substack{c|N \\
			\gcd(c,N/c)|\gcd(N/f_\chi, q^{2\nu+1-2j}-1)}}
	\varphi(\gcd(c,N/c))\chi_c(q)^j\chi_c'(q)^{2\nu+1-j} \zeta^{2\nu+1},
\end{align*}
which is transformed into
\begin{align*}
	-\sum_{\nu=0}^{\infty}\sum_{\substack{l\in \NN \\ l|(N/f_\chi)}} \sum_{\substack{c | N \\
			l=\gcd(c,N/c)}}\sum_{\substack{0\le j \le \nu \\
			2\nu+1-2j\equiv 0 (\text{mod }m_{q,l})}} \varphi(l)
	\{q^{k-1}\chi(q)\chi_c'(q^{-2})\}^j \{\chi_c'(q)\zeta\}^{2\nu+1}.
\end{align*}
By changing the order of summations, this series equals
$$-\sum_{l|(N/f_\chi)}\varphi(l)\sum_{\substack{c | N \\ l=\gcd(c,N/c)}}
\sum_{\nu=0}^\infty \{\chi_c'(q)\zeta\}^{2\nu+1}
\sum_{\substack{0\le j \le \nu \\
		2\nu+1-2j\equiv 0 (\text{mod }m_{q,l})}}\{q^{k-1}\chi(q)\chi_c'(q^{-2})\}^j.
$$
If there exists $w \in \ZZ$ such that $2\nu+1-2j=m_{q,l}w$,
then $m_{q,l}$ and $w$ are odd and satisfy $m_{q,l}w\in\{1,3,5,\ldots,2\nu+1\}$.
Set $n_{q,l}=\frac{m_{q,l}-1}{2} \in \ZZ_{\ge0}$.
When $t\in\NN$ and $m_{q,l}t+n_{q,l}\le \nu <m_{q,l}(t+1)+n_{q,l}$,
$w$ ranges so that $w =2s+1$ with $0\le s \le t$.
Thus, the series as above equals
{\allowdisplaybreaks\begin{align}
	&-\sum_{\substack{l|(N/f_\chi) \\
			m_{q,l}: {\rm odd}}}\varphi(l)\sum_{\substack{c | N \\ l=\gcd(c,N/c)}}
	\sum_{t=0}^\infty
	\sum_{\nu=m_{q,l}t+n_{q,l}}^{m_{q,l}(t+1)+n_{q,l}-1}
	\{\chi_c'(q)\zeta\}^{2\nu+1}
	\sum_{\substack{s =0}}^{t}\{q^{k-1}\chi(q)\chi_c'(q^{-2})\}^{\nu-n_{q,l}-m_{q,l}s} \notag \\
	= &
	-\sum_{\substack{l|(N/f_\chi) \\
			m_{q,l} : {\rm odd}}}\varphi(l)\sum_{\substack{c | N \\ l=\gcd(c,N/c)}}
	\sum_{t=0}^\infty
	\frac{\{q^{k-1}\chi(q)\zeta^2\}^{m_{q,l}t+n_{q,l}}
		(1-\{q^{k-1}\chi(q)\zeta^2\}^{m_{q,l}}) }{1-q^{k-1}\chi(q)\zeta^2} \notag\\
	& \times \frac{1-\{q^{k-1}\chi(q)\chi_c'(q^{-2})\}^{-m_{q,l}(t+1)}}{1-\{q^{k-1}\chi(q)\chi_c'(q^{-2})\}^{-m_{q,l}}} \{q^{k-1}\chi(q)\chi_c'(q^{-2})\}^{-n_{q,l}}\chi_c'(q)\zeta \notag \\
	= &
	-\sum_{\substack{l|(N/f_\chi) \\
			m_{q,l} : {\rm odd}}}\varphi(l)\sum_{\substack{c | N \\ l=\gcd(c,N/c)}}
	\sum_{t=0}^\infty
	(\{q^{k-1}\chi(q)\zeta^2\}^{m_{q,l}t}- \{\chi_c'(q^2)\zeta^2\}^{m_{q,l}t}\{q^{k-1}\chi(q)\chi_c'(q^{-2})\}^{-m_{q,l}})\notag \\
	& \times	\frac{1-\{q^{k-1}\chi(q)\zeta^2\}^{m_{q,l}}}{1-q^{k-1}\chi(q)\zeta^2} \frac{\{\chi_c'(q)\zeta\}^{2n_{q,l}+1}}{1-\{q^{k-1}\chi(q)\chi_c'(q^{-2})\}^{-m_{q,l}}} \notag \\
	= &
	-\sum_{\substack{l|(N/f_\chi) \\
			m_{q,l} : {\rm odd}}}\varphi(l)\sum_{\substack{c | N \\ l=\gcd(c,N/c)}}
\left(\frac{1}{1-\{\chi(q)q^{k-1}\zeta^2\}^{m_{q,l}}}-\frac{\{q^{k-1}\chi(q)\chi_c'(q^{-2})\}^{-m_{q,l}}}{1-\{\chi_c'(q)^2\zeta^2\}^{m_{q,l}}}  \right)\notag \\
	& \times \frac{1-\{q^{k-1}\chi(q)\zeta^2\}^{m_{q,l}}}{1-q^{k-1}\chi(q)\zeta^2} \frac{\{\chi_c'(q)\zeta\}^{m_{q,l}} }{1-\{q^{k-1}\chi(q)\chi_c'(q^{-2})\}^{-m_{q,l}}} \notag \\
	=& -\sum_{\substack{l|(N/f_\chi) \\
			m_{q,l} : {\rm odd}}}\varphi(l)\sum_{\substack{c | N \\ l=\gcd(c,N/c)}}
	 \frac{\{\chi_c'(q)\zeta\}^{m_{q,l}}}{(1-\chi(q)q^{k-1}\zeta^2)(1-\{\chi_c'(q^2)\zeta^2\}^{m_{q,l}})}, \label{first-formula}
\end{align}
}where $X=\chi(q)^{1/2}q^{(k-1)/2}\zeta$. This computation is justified when $|\zeta|<q^{(1-k)/2}$, or equivalently
$|X|<1$.

In a similar fashion, the term
\eqref{second} equals
$$-\sum_{l|(N/f_\chi)}\varphi(l)\sum_{\substack{c| N \\ l=\gcd(c,N/c)}}
\sum_{\nu=0}^\infty \{\chi_c'(q)\zeta\}^{2\nu}
\sum_{\substack{0\le j \le \nu \\
		2\nu-2j\equiv 0 (\text{mod }m_{q,l})}}\{q^{k-1}\chi(q)\chi_c'(q^{-2})\}^j.
$$
When $2\nu-2j=m_{q,l}w$ for $w \in \ZZ$,
then $m_{q,l}w$ is even and $j=\nu-m_{q,l}w/2$.
Suppose that $m_{q,l}$ is even and that $2^{-1}m_{q,l}t\le \nu <2^{-1}m_{q,l}(t+1)$ for $t \in \ZZ_{\ge 0}$.
Then, $w$ varies so that $0\le w \le t$.
The term for even $m_{q,l}$ is given by
{\allowdisplaybreaks \begin{align}
&-\sum_{\substack{l|(N/f_\chi) \\
		m_{q,l} : {\rm even}}}\varphi(l)\sum_{\substack{c|N \\ l=\gcd(c,N/c)}}
\sum_{t=0}^{\infty}\sum_{\nu=2^{-1}m_{q,l}t}^{2^{-1}m_{q,l}(t+1)-1}\{\chi_c'(q)\zeta\}^{2\nu}
\sum_{w=0}^{t}\{q^{k-1}\chi(q)\chi_c'(q^{-2})\}^{\nu-m_{q,l}w/2}\notag \\
= & -\sum_{\substack{l|(N/f_\chi) \\
		m_{q,l} : {\rm even}}}\varphi(l)\sum_{\substack{c| N \\ l=\gcd(c,N/c)}}\frac{1}{(1-q^{k-1}\chi(q)\zeta^2)(1-\{\chi_c'(q)\zeta\}^{m_{q,l}})}. \label{second-m-even}
\end{align}
}When $m_{q,l}$ is odd, $w$ must be $w=2s$ with $0\le s \le t$.
Hence, the term for odd $m_{q,l}$ is given by
{\allowdisplaybreaks \begin{align}
&-\sum_{\substack{l|(N/f_\chi) \\
		m_{q,l} : {\rm odd}}}\varphi(l)\sum_{\substack{c| N \\ l=\gcd(c,N/c)}}
\sum_{t=0}^\infty
\sum_{\nu=m_{q,l}t}^{m_{q,l}(t+1)-1}
\{\chi_c'(q)\zeta\}^{2\nu}
\sum_{\substack{s =0}}^{t}\{q^{k-1}\chi(q)\chi_c'(q^{-2})\}^{\nu-m_{q,l}s}\notag \\
=	& -\sum_{\substack{l|(N/f_\chi) \\
		m_{q,l} : {\rm odd}}}\varphi(l)\sum_{\substack{c| N \\ l=\gcd(c,N/c)}}
\frac{1}{(1-q^{k-1}\chi(q)\zeta^2)(1-\{\chi_c'(q^2)\zeta^2\}^{m_{q,l}})}.\label{second-m-odd}
\end{align}
}By the consideration so far, we obtain the assertion from \eqref{third-formula}, \eqref{first-formula},
\eqref{second-m-even} and \eqref{second-m-odd}.
\end{proof}

%
%\medskip
%\noindent
%{\bf Remark} :
%If $N$ is square-free and $\chi$ is principal, then the series $\sum_{\nu=0}^{\%infty}A_3(q^\nu) \zeta^{\nu}$ has the following simple expression
%$$-d(N)
%\frac{1+q^{(1-k)/2}X}{2(1-X^2)(1-q^{(1-k)/2}X)}, \qquad |X|<1.$$
%
%\medskip
%\noindent

%%%%%%%%%%%%%%%%%%%%%%%%%%%%%%%%%%%%%%%%%%%%%%%%%%%%%%%%%%%%%%%%%%%%%%%%%%%%%%
\section{Proof of Theorem \ref{best possible} and Corollary \ref{equidist}}
%%%%%%%%%%%%%%%%%%%%%%%%%%%%%%%%%%%%%%%%%%%%%%%%%%%%%%%%%%%%%%%%%%%%%%%%%%%%%%

By virtue of the resolvent trace formula (Theorem \ref{resolventTF}), we have the following estimate.
\begin{cor}\label{esti}
	For any $\e>0$ and $k\ge 4$, we have the bound
	\begin{align*}|A_{k,q}(\nu)|\ll_{q, k,\e} q^{\nu(1/2+\e)}, \quad \nu\in \N.
	\end{align*}
\end{cor}
\begin{proof}
	By the Ramanujan bound, the left-hand side of \eqref{RTF}
	in Theroem \ref{resolventTF} is holomorphic on the unit disc $|X|<1$. The first
	three summands of the right-hand side are evidently holomorphic on $|X|<1$. 
	The fourth summand is vanishing by $k\ge 4$.
	Hence, the power series $\sum_{\nu=0}^{\infty}A_{k,q}(\nu)(\chi(q)^{-1/2}q^{-1/2}\,X)^{\nu}$ (which is convergent for $|X|\ll 1)$ becomes a holomorphic function on the unit disc $|X|<1$.
	The Cauchy estimate yields the desired bound.
%	($(\sum_{t}H()^2)^{1/2}\ll q^{\nu(1/2+\e)}$)
\end{proof}

Let us prove Theorem \ref{best possible}.
The estimate $({\rm E}_{n, 1/2})$ for any $n\in\NN$ follows from Corollary \ref{esti} applied to the full modular group $\SL_2(\Z)$ with $\chi={\bf 1}$; we note that Chebyshev polynomials form a basis of the space of polynomials. By \cite[Th\'eor\`eme 1]{Serre}, there exist a large $k \in 2\ZZ_{\ge 2}$ and a normalized Hecke eigenform $f_i$ of weight $k$ such that $\a_{i}(q)\neq 1$.
Then,
$\sum_{\nu=0}^{\infty}A_{k,q}(\nu)q^{-\nu/2}X^\nu$ has a pole at $X=\a_{i}(q)$
by \eqref{RTF}, which is a special case of Theorem \ref{resolventTF}. Hence, its radius of convergence is $1$. Therefore for any $\a\in (0,1/2)$, the estimate $({\rm E}_{n,\a})$ does not hold for some $n \in \NN$.
\qed

\medskip
\noindent

Let us define a Radon measure $\mu_\nu^{N,\chi,q}$ on $[-1,1]$ by
$$\langle \mu_{\nu}^{N,\chi, q}, f\rangle = 2^{-1}q^{-\nu}\sum_{t \in \ZZ, \ t^2<4q^\nu}H_{N,\chi}(4q^\nu-t^2)\delta_{\frac{t}{2q^{\nu/2}}}.$$
Then, $\mu_{\nu}^{N,\chi,q}$ is positive if $\chi$ is principal.
\begin{lem}\label{formula for 1}We have
	$$\langle \mu_{\nu}^{N,\chi,q}, 1\rangle = \delta_{\chi,\bf1}\frac{1}{1-q^{-1}}+\Ocal_{N, \chi}(\nu q^{-\nu/2}).$$
\end{lem}
\begin{proof}By the Eichler-Selberg trace formula (Theorem \ref{ES trace}) for $k=2$, we have the equality
	$$\frac{1}{2}\sum_{t \in \ZZ, \ t^2<4q^\nu}H_{N,\chi}(4q^\nu-t^2) = A_1(q^\nu)+A_3(q^\nu)+A_4(q^\nu)-\tr(T(q^\nu)|S_2(N,\chi)).$$
	By easy calculation, we have two equalities $A_4(q^\nu)=\delta_{\chi,\bf1}\frac{q^{\nu+1}-1}{q-1}$ and
	$A_1(q^\nu) = \tilde{\chi}(\sqrt{q^\nu})\frac{\psi(N)}{12}$.
	The term $A_3(q^\nu)$ is estimated as
	$$|A_3(q^\nu)|\ll \sum_{j=0}^{\nu}\min(q^\nu, q^{\nu-j})\sum_{c|N}\varphi(\gcd(c,N/c))\ll \nu q^{\nu/2}d(N)\varphi(N).$$
	By the Ramanujan bound $|a_{f_i}(n)|\le d(n)n^{1/2}$, we have
	$$|\tr(T(q^\nu)|S_2(N,\chi))| \le \sum_{i=0}^{\dim S_2(N,\chi)} d(q^\nu)q^{\nu/2} = (\nu+1)q^{\nu/2} \dim S_{2}(N,\chi).$$
	By combining the evaluations as above, the value $\langle \mu_{\nu}^{N,\chi, q}, 1\rangle$, which is equal to $q^{-\nu}\{A_1(q^\nu)+A_3(q^\nu)+A_4(q^\nu)-\tr(T(q^\nu)|S_2(N,\chi))\}$ is estimated by $q^{-\nu}A_4(q^\nu)+q^{-\nu}\Ocal_{N,\chi}(\nu q^{\nu/2})$. Thus we are done.
\end{proof}

\begin{thm} \label{weak conv}
	As $\nu\rightarrow \infty$, the measure $\mu_{\nu}^{N,\chi , q}$ converges $*$-weakly to the measure
	$$
	\delta_{\chi,\bf1}\frac{1}{1-q^{-1}}\,\frac{2}{\pi}\sqrt{1-x^2}\d x.
	$$
\end{thm}
\begin{proof}Let us consider the case where $\chi$ is principal.
	By the positivity of $\mu_\nu^{N,{\bf1},q}$ and \cite[Proposition 2]{Serre},
	%$\{U_n(x)\}_{n=0}^\infty$ is a complete orthonormal basis of $L^{2}([-1,1]; \tfrac{2}{\pi}\sqrt{1-x^2}\d x)$,
	it suffices to show the following convergence:
	\begin{align}
	&\lim_{\nu\rightarrow \infty} \langle \mu_{\nu}^{N, {\bf1}, q},U_{n}\rangle=0, \quad (n>0) \label{mu-U}, \\
	&\lim_{\nu\rightarrow \infty} \langle \mu_{\nu}^{N,{\bf1}, q},1\rangle=\frac{1}{1-q^{-1}}.
	\label{mu-1}
	\end{align}
	Corollary~\ref{esti} gives us the bound $|\langle \mu_{\nu}^{N,{\bf1}, q},U_{k-2}\rangle|\ll_{\e}q^{(-1/2+\e)\nu}$ which is sufficient to have \eqref{mu-U}. Furthermore, Lemma~\ref{formula for 1} implies \eqref{mu-1}. This completes the proof of the assertion when $\chi$ is principal.

In the non-principal case of $\chi$,
although $\mu_\nu^{N, \chi, q}$ is not always positive,
we obtain the assertion in the same way as the principal case
from the inequality
$$\langle \mu_{\nu}^{N,\chi,q}, f\rangle \le \mu_{\nu}^{N,{\bf1},q}(1)\sup_{x\in[-1,1]}|f(x)|, \quad f \in C([-1,1])$$
and the vanishing of $\lim_{\nu\rightarrow \infty}\mu_\nu^{N,\chi,q}$ on the polynomial functions (Lemma \ref{formula for 1}).
\end{proof}

\medskip
\noindent
{\bf Remark} : To have the convergence of the measure, any improvement of the Hecke bound of the Fourier coefficients is sufficient (see the proof of Lemma \ref{formula for 1}).

\medskip
By the equalities $\mu_{q^\nu}=2q^{\nu}\,\mu_{\nu}^{1,{\bf1},q}$ and $$\mu_{q^\nu}([-1,1])=2\s(q^\nu)-\sum_{j=0}^{\nu}\min(q^j,q^{\nu-j})\sim 2q^\nu(1-q^{-1})^{-1}, \qquad (\nu \rightarrow \infty),$$
Theorem \ref{equidist} follows from Theorem \ref{weak conv} and \cite[Proposition 1]{Serre}.
\qed

With a bit more work, we have a generalization of Theorem \ref{equidist} as follows. Let $S=\{q_1,\dots,q_r\}$ be a finite set of prime numbers and $\N(S)$ the set of those positive integers whose prime divisors belong to $S$. Then the limit formula in Theorem~\ref{equidist} is true when $m$ inside $\N(S)$ grows to infinity, i.e., for any $[\a,\b]\subset [-1,1]$,  
\begin{align*}
\lim_{\substack{\nu_q \rightarrow \infty (\forall q \in S) \\ m=\prod_{q \in S}q^{\nu_q}\in \N(S)}} \frac{\mu_{m}([\a,\b])}{\mu_{m}([-1,-1])}=\frac{2}{\pi}\int_{\a}^{\b}\sqrt{1-x^2}\d x. 
\end{align*}
This is shown by the same argument as above by establishing the estimate 
\begin{align}
|\langle \mu_{m},U_{l}\rangle|\ll_{l,\e} m^{1/2+\e}, \quad m\in \N(S)
 \label{Remark1}
\end{align}
for each $l\in 2\N$. This bound in turn is shown by the multivariable analogue of \eqref{RTF1}
for $k\ge 4$, which equates the rational function 
\begin{align*}
\sum_{i=1}^{\dim S_k(\SL_2(\ZZ))}\prod_{j=1}^{r} \frac{-q_j^{(1-k)/2}}{q_j^{(1-k)/2}a_i(q_j)-(X_j+X_j^{-1})}
\end{align*}
with a sum of the following three power series: 
\begin{align}
&\tfrac{k-1}{12}
\prod_{j=1}^{r}\frac{q_j^{\frac{1-k}{2}}X_j}{1-q_j^{-1}X_j^2}, 
 \notag
\\
& \sum_{(\nu_1,\dots,\nu_r) \in \Z_{\ge 0}^{S}} A_3(q_1^{\nu_1}\cdots q_r^{\nu_r})
\prod_{j=1}^{r}(q_j^{(1-k)/2}X_j)^{\nu_j+1}, 
 \label{Hypterm} \\
& -\tfrac{1}{2}\sum_{(\nu_1,\dots,\nu_r) \in \Z_{\ge 0}^{S}} 
\langle \mu_{q_1^{\nu_1}\cdots q_r^{\nu_r}},U_{k-2}\rangle \prod_{j=1}^{r}\{(q_j^{-1/2}X_j)^{\nu_j}
\,q_j^{(1-k)/2}X_j\}.
\label{Ellterm}
\end{align}
It is not clear that the series \eqref{Hypterm} has a rational expression; but, it is easy to see its absolute convergence on the polydisc $\Delta=\{(X_1,\dots,X_r)\in \C^r|\,\max_{1\leq j \leq r}|X_j|<1\}$.
Hence invoking the Ramanujan bound, the holomorphic function defined in a neighborhood of the origin by the convergent multi-series \eqref{Ellterm} has a holomorphic continuation to the polydisc $\Delta$. Then the Cauchy estimate yields the bound \eqref{Remark1}. Details are left to the readers.

%
%
%\begin{lem} \label{L4} (Hurwitz's formula)We have
%\begin{align*}
%\sum_{\substack{t\in \Z \\ t^2<4m}} H(4m-t^2)=2\sigma(m)-\sum_{0<d|n}\min(d,n/d),\end{align*}
%where $\s(m)$ is the sum of positive divisors of $m$. 
%\end{lem}
%\begin{proof} This should follow from the trace formula of $T(m)$ on $S_2(\SL_2(\Z)=\{0\}$. See \cite[p.236]{Cohen}, \cite[Lemma 16]{MurtySinha}. 
%\end{proof}

%\medskip
%\noindent
%{\bf QUESTION}: Can one make Theorem~\ref{L7} effective as \cite[Theorem 2]{MurtySinha} did for Serre's theorem. 

%%%%%%%%%%%%%%%%%%%%%%%%%%%%
\section*{Aknowledgements}
%%%%%%%%%%%%%%%%%%%%%%%%%%%%

The second author was supported by Grant-in-Aid for Scientific research (C) 15K04795.

\medskip
\noindent
{Shingo SUGIYAMA\\
	Department of Mathematics, College of Science and Technology, Nihon University,
	Suruga-Dai, Kanda, Chiyoda, Tokyo 101-8308, Japan} \\
{\it E-mail} : {\tt s-sugiyama@math.cst.nihon-u.ac.jp}

\medskip
\noindent
{Masao TSUZUKI \\ Department of Science and Technology, Sophia University, Kioi-cho 7-1 Chiyoda-ku Tokyo, 102-8554, Japan} \\
{\it E-mail} : {\tt m-tsuduk@sophia.ac.jp}

\end{document}